\documentclass{amsart}

\usepackage{amsmath,amsfonts,amssymb}
\usepackage{graphicx,psfrag,subfigure}
\newtheorem{thm}{Theorem}
\newtheorem{rk}{Remark}
\newtheorem{prop}{Proposition}
\newtheorem{clly}{Corollary}
\newtheorem{lemma}{Lemma}
\newtheorem{defi}{Definition}

\newcommand{\R}{{\mathbb{R}}}

\newcommand{\Z}{{\mathbb{Z}}}
\newcommand{\N}{{\mathbb{N}}}

\newcommand{\D}{{\mathbb{D}}}
\newcommand{\fix}{Fix}
\newcommand{\out}{out}
\newcommand{\inn}{inn}
\newcommand{\fil}{Fill}

\newcommand{\ind}{Ind}

\title{On the growth rate inequality for periodic points in the two sphere}
\author{G. Honorato, J. Iglesias, A. Portela, \'A. Rovella, F. Valenzuela and J. Xavier}

\begin{document}
\maketitle
 \section{Introduction}

This paper deals with the following open problem:  let $f: S^2 \to S^2$ be a continuous map of degree $d$, $|d|>1$, and let $N_nf$ denote the number of fixed points of $f^n$. When does 
the growth rate inequality

\begin{equation}\label{tasa}
\limsup \frac{1}{n} \log N_nf\geq \log d.
\end{equation}
\noindent  hold for $f$? (This is Problem 3 posed in \cite{shub2}).

In the case equation (\ref{tasa}) holds, we say that $f$ 
{\em has the rate}.\\ There are many counterexamples to this. The simplest one is the map expressed in polar coordinates in $\R^2$ as 
$(r,\theta)\to (dr,d\theta)$ and extended to the sphere with $\infty\to \infty$. It has degree $d$ and just two periodic points.  More interesting examples (where
the nonwandering set is not reduced to the set of periodic points) can be found in \cite{iprx2}.  On the other hand, the inequality is known to hold if $f$ is a rational map \cite{shub}, 
if $f$ is a branched covering  having a 
completely invariant simply connected region $R$ with locally connected boundary without only one critical fixed point in $\partial R$ \cite{bcs},
if $f$ is $C^1$ and preserves
the latitude foliation \cite{ps} and \cite{mis}, if $f$ is $C^1$ and longitudinal \cite{gmn}, if the critical points form a two-periodic cycle \cite{iprx3}, and if all periodic orbits are
isolated as invariant sets and $f$ has no sources of degree
$r$, $|r|\geq 1$ \cite{lhc}.  \\

We prove the following:\\

\noindent
{\bf Theorem A.}
{\em Let $f:S^2\to S^2$ be a continuous map such that $\deg f = d, |d|>1$.  Suppose $f$ has two attracting fixed points denoted $N$ and $S$ and let $A=S^2\setminus \{N,S\}$.
Assume that if a loop $\gamma\subset f^{-1}(A)$ is homotopically trivial in $A$, then $f(\gamma)$ is also 
homotopically trivial in $A$.
Then $f$ has the rate.}\\

Note that $A$ is an annulus whose preimage under $f$ 
is contained in $A$ but whose image under $f$ is not necessarily contained in $A$ because $N$ and $S$ may have other preimages appart from themselves. 
This result generalizes Pugh-Shub's theorem in \cite{ps} and also the bidimensional result in \cite{gmn}, as if $f$ preserves the latitude or longitudinal foliation, then the hypothesis of
Theorem A is trivially satisfied, as any loop 
loop $\gamma\subset f^{-1}(A)$  homotopically trivial in $A$ is also homotopically trivial in $f^{-1} (A)$.  If in addition $f$ is $C^1$ then the hypothesis on the fixed attractors 
(maybe for $f^2$)
is also satisfied (namely, the poles). In this way the strong hypothesis of preserving a foliation is dropped, replaced by an assumption on the homological action of the map 
$f: |_{f^{-1} (A)}: f^{-1} (A)\to A$, relating the long term dynamics of the endomorphism and its action on the algebraic topology.

As an example  not satisfying the hypothesis of Theorem A, even when having the two fixed attractors, consider $z\to z^2+c$ where 
$c\neq 0$ is small so that there is a fixed attracting point $S$ close to $0$ that is not critical. Then $N=\infty$ and $S$ are the unique fixed attractors 
and the preimage $U$ of $A$ equals the sphere minus three points $N$, $S$ and the preimage of $S$ that is not $S$, denote it $S'$. Now, a small Jordan curve
$ \gamma$ surrounding $S'$ is a curve in $U$ inessential in $A$ but $f(\gamma)$ is essential in $A$.  Note however that this example has the rate as all rational function do.

The proof of Theorem A relies in Nielsen Theory and Lefschetz index Theorem following previous work in \cite{iprx2} and \cite{iprx3}.

\

\section{Straightening $f$}

Recall that $A$ stands for the annulus $S^2\setminus\{N,S\}$ and $f:S^2\to S^2$is a continuous map for which we assume two additional hypothesis:

(H1) $\gamma\subset A\cap f^{-1}(A)$ is trivial in $A$ implies $f(\gamma)$ is trivial in $A$.\\
(H2) The points $N$ and $S$ are fixed points for $f$. They are also weak attractors, meaning that there exist disjoint open sets $V_N$ and $V_S$ such that $N\in V_N$, $S\in V_S$ and 
$f(\overline{V_N})\subset V_N$, $f(\overline{V_S})\subset V_S$ .

In this section we give some topological preliminaries, proving that there exists a map $g$ homotopic to $f$ such that each component of the preimage of $N$ and $S$ is either $N$, or $S$, or an essential set in $A$. A subset of $A$ is essential if it separates $N$ and $S$. A characterization is: a closed set $K\subset A$ is not essential (called {\em inessential}) iff there exists an arc joining $N$ and $S$ disjoint from $K$.\\

 For a closed plane curve $\gamma:[0,1]\to \R^2\setminus\{p\}$ define the index of $\gamma$ with respect to $p$ as 
$$
\ind_\gamma(p)=\frac{1}{2\pi i}(\tilde{\gamma}(1)-\tilde\gamma(0)) 
$$
where $\tilde \gamma$ is any lift of $\gamma$ under the covering map $E:\R^2\to \R^2\setminus\{p\}$ given in complex form as $E(z)=\exp(z)+p$.

Also denote by $\out(\gamma)$ the set of points $p$ such that $\ind_\gamma(p)=0$ and by $\inn(\gamma)$ the set of points $p$ such that 
$\ind_\gamma(p)\neq 0$.
Note that a point in a bounded component of the complement of the curve $\gamma$ may have zero index. A fact that characterizes when a point 
$p$ belongs to $\out(\gamma)$ is the existence of a homotopy of curves $\gamma_t:S^1\to\R^2\setminus\{p\}$ such that $\gamma_0=\gamma$ and $\gamma_1$ is constant.

Now take a curve $\gamma$ contained in $A$, and think of $\gamma$ as a curve in the plane, via the homeomorphism $S^2\setminus\{N\}\to \R^2$. 
Then $\gamma$ is homotopically trivial in $A$ iff $S\in\out(\gamma)$.
If $\gamma$ is inessential it does not matter if we choose $S$ instead of $N$ and consider $\gamma$ as a subset of $S^2\setminus\{S\}$, because the index changes the sign but not its absolute value, so $\inn(\gamma)$ remains the same set.

The purpose of this section is to remove inessential components of the preimage of $N$ and $S$. To give an idea of what follows, 
assume that $N'\neq N$ is an isolated point in the preimage of $N$. $A$ is considered as the annulus $\R^2\setminus \{N\}$ where $\R^2=S^2\setminus\{S\}$, so the index of a curve is well defined. Take a simple closed curve $\gamma$ surrounding $N'$, that is, such that $\ind_\gamma(N')=1$. If $\gamma$ is sufficiently close to $N'$, then $\inn(\gamma)$ does not intersect preimages of $N$ or $S$ other than $N'$ itself. 
Then hypothesis (H1) implies that $f(\gamma)$ is a curve inessential in $A$. 
This means that $N\in\out(f(\gamma))$, which immediately implies (see Corollary \ref{c1} below), that there exists $g$ homotopic to $f$ such that $g=f$ in $\out(\gamma)$ and $g(\inn(\gamma))$ does not contain $N$. 
In other words, $g(\inn(\gamma))\subset A$. 

There is a problem, however: the set of periodic points may change if this is done without  extra care. Note that as $N$ is a weak attractor, then the map $g$ mentioned above may be different from $f$ only in a set of points close to the preimage $N'$ of $N$. It is easy to take $g$ such that $f(x)\neq g(x)$ implies that $g(x)$ is close enough to $N$ to assure that $g^n(x)\in V_N$. This shows that the procedure above does not create new periodic points outside $V_N$. 

This was just the idea of the procedure to remove {\em isolated} preimages of $N$.

$\D$ denotes the open unit disc in the plane, $\overline \D$ its closure and $S^1 = \partial \D$.
 
\begin{lemma}
\label{l1}
Let $\gamma:S^1\to \R^2$ be a closed curve and $p$ a point in $\out(\gamma)$. 
Then there exists a continuous $F=F_{\gamma,p}:\overline\D\to\R^2$ such that 
$F|_{S^1}=\gamma$ and $p \notin F(\D)$.
\end{lemma}
\begin{proof}  Recall that as $p$ belongs to $\out(\gamma)$ there exists a homotopy of curves $\gamma_t:S^1\to\R^2\setminus\{p\}$ such that $\gamma_0=\gamma$ and $\gamma_1$ is constant.
 For $t\in[0,1]$ and $\theta\in S^1$, take $F(t,\theta)=\gamma_{1-t}(\theta)$.
\end{proof}

\begin{lemma}
\label{l2}
Let $f:\overline{\D}\to \R^2$ be a continuous map; define $\gamma=f|_{S^1}$ and let $p$ be a point in  $\out(\gamma)$. Then there exists $g:\overline{\D}\to \R^2$ homotopic to $f$ such that $g=f$ in $S^1$ and $p\notin g(\D)$.
\end{lemma}
\begin{proof} The homotopy transits the segment joining $f(x)$ to $F_{\gamma,p}(x)$ for each $x\in\D$, in other words, $g=F_{\gamma,p}$.
\end{proof}

The union of a set $K\subset A$ (no matter if it is essential or not) with those components of its complement not containing $N$ or $S$ is called the fill of $K$ and denoted $\fil(K)$.

Note that when $\gamma:S^1\to A$ is inessential and injective (i.e {\em simple}), then $\inn(\gamma)=\fil(\gamma)$. But this is trivially false if $\gamma$ is not simple.

\begin{clly}
\label{c1}
Let $\gamma:S^1\to A$ be homotopically trivial, simple and contained in $f^{-1}(A)$ (so that $f(\gamma)$ is inessential in $A$). 
If $S\notin f(\inn(\gamma))$, then there exists $g$ homotopic to $f$ such that $g=f$ in 
$\out(\gamma)$ and $g(\inn(\gamma))\subset A$.
\end{clly}
\begin{proof} There is a first case where $N\notin f(\inn(\gamma))$, in this case the conclusion trivially holds for $f$, no need of the homotopy. So the problem is when $N$ has $f$-preimages in $\inn(\gamma)$. The plane is taken now as $\R^2=S^2\setminus \{S\}$.
As $\gamma$ is simple, the closure of $\inn(\gamma)$ is homeomorphic to $\overline \D$. 
The hypothesis implies that $f(\overline \D)$ is contained in $\R^2$.
Then the conclusion is a direct application of Lemma \ref{l2}, taking as $p$ the point $N$.
\end{proof}

Note that there exists a number $\delta>0$ such that $d(x,y)\geq \delta$ whenever $f(x)=N$ and $f(y)=S$. In addition, there exists $\rho$ 
such that $d(x,y)<\rho$ and $f(x)=N$ implies $f(y)\in V_N$. Identical properties hold with $S$ instead of $N$, same $\rho$.\\

Cover $f^{-1}(\{N,S\})$ with a finite number of balls of diameters less than $\rho$ and $\delta/2$. Assume every ball intersects $f^{-1}(\{N,S\})$. Denote by $U$ the union of the balls. Note that by the choice 
of $\delta$, if a component $C$ of $U$ intersects $f^{-1}(N)$ then it cannot intersect $f^{-1}(S)$. Moreover, by the choice of $\rho$, 
it follows that $f(U)$ is contained in $V_N\cup V_S$. Note, however, that if $C$ is a component of $U$, then it may occur that $f(\fil(C))$ is not contained in $V_N\cup V_S$, it may be equal to the whole sphere.

The number of components of $U$ is finite and all its boundary components are simple closed curves (perhaps the radius of the balls has to be modified a little in order to obtain this property). For simplicity say that a component of $U$ has type I if it contains $N$ or $S$, has type II if it is essential in $A$ and type III if it is inessential.
This is the main result of this section:

\begin{prop}
\label{p1}
There exists $g$ homotopic to $f$ such that $g^{-1}(\{N, S\})$ has finitely many components, each of which is either essential in $A$ or it contains $N$ or $S$. 
Moreover, $g(x)\neq f(x)$ implies $g(x)\in V_N \cup V_S$.
\end{prop}
\begin{proof} It must be shown that each inessential component of $f^{-1}(\{N,S\}$ can be removed by a homotopic map; this will be done by first removing inessential components of $U$, using Corollary \ref{c1}.\\
{\em Step 1}. Let $C$ be a component of $U$ of type (III) and let $\gamma$ be the simple closed curve that bounds $\fil(C)=\inn(\gamma)$. 
Say that $N\in f(C)$ (so that $S\notin f(C)$). Assume first that $S\notin f(\fil(C))$. By Corollary \ref{c1} there exists a map $g$, homotopic to $f$, such that $g(x)=f(x)$ for every $x\notin \fil(C)$ and $N\notin g(\fil(C))$. Moreover, as $f(\gamma)\subset V_N$, then it can also be imposed to $g$ that $g(\fil(C))$ is contained in $V_N$. (This shows how to proceed when $C$ is a component of $U$ such that $f(\fil(C))$ does not contain both $N$ and $S$).\\
Now assume that there is a component $C'$ of $f^{-1}(S)$ contained in $\fil(C)$, so that $S\in f(\fil(C))$ and the above standing hypothesis ceases to hold. In addition, there can also be a component of $f^{-1}(N)$ contained in $\fil(C')$, and so on. But as the number of components is finite, and the above argument shows how to remove a component not containing preimages of $N$ or $S$ in its $\fil$, we conclude that in a finite number of steps, all components of type (III) will disappear. We have thus obtained a map $g$ homotopic to $f$ that is free of type (III) components. Note also that 
$f(x)\neq g(x)$ implies $g(x)\in V_N\cup V_S$, and $g=f$ in a neighborhood of $\{N,S\}$.\\
{\em Step 2.} By the previous step we may assume that there are no inessential components of $U$. Let $C$ be a component of type (II), assume $N\in f(C)$. Let $C'=\fil(C)$ and note that $C'$ is a compact subset of $A$ such that 
$S\notin f(C')$, because in $\fil(C)$ a component containing a preimage of $S$ has to be inessential, so that it was already removed in the previous step. Note that $C'$ is a compact essential annulus contained in $A$. It is clear that there exists $g$ homotopic to $f$ such that the following properties hold: (1) $g=f$ outside $C'$, (2) $g^{-1}(N)$ is an essential annulus (or circle) contained in the interior of $C'$ and (3) $g(x)\neq f(x)$ implies that $g(x)\in V_N$.\\
{\em Step 3.} Let $C$ be the component of $U$ containing $N$. As there are no preimages of $S$ in $C'=\fil(C)$ of $U$ (same reasons as in step 2) and  $C'$ is a disc, there exists a 
map $g$ homotopic to $f$ such that: (1) $g=f$ outside $C'$, (2) $g^{-1}(N)$ is a disc contained in the interior of $C'$ and (3) $g(x)\neq f(x)$ implies $g(x)\in V_N$.
\end{proof} 

With this proposition in hand we can trivially proceed to further homotopies to obtain that the components of type I of $g$ are $N$ and $S$, that the components
of type II are circles and that there are no components of type III. Therefore, the components of $g^{-1}(\{N, S\})$ are $N$, $S$ and essential circles. Moreover, 
the boundary of a component $A'$ of $g^{-1}(A)$ has two components, each contained in the preimage of $N$ or $S$. If both boundary components have the same image, say $N$, then
there exists a map $g'$ homotopic to $g$ so that $g'(A')=N$. Afterwards, a last homotopy is performed in order to transform this annulus into an essential circle. With these
considerations in mind, we have thus obtained the following:\\
\noindent
{\bf Conclusion.} There exists $g$ homotopic to $f$ that satisfies the following properties:\\
(H1) $N$ and $S$ are fixed points of $g$, with weak attracting neighborhoods $V_N$ and $V_S$.\\
(H2) The component of $f^{-1}(N)$ containing $N$ is $\{N\}$. Idem for $S$.\\
(H3) Any other component of the preimage of $N$ is an essential circle in $A$. Idem for $S$. \\
(H4) If $g(x)\neq f(x)$ then $g(x)\in V_N\cup V_S$. It follows that every periodic cycle of $g$ not intersecting $V_N\cup V_S$ is also a periodic cycle of $f$.\\
(H5) Each component of $g^{-1}(A)$ is an annulus, one of whose boundaries component is preimage of $N$ and the other preimage of $S$.\\

Note that the hypothesis (H) of the statement of the Theorem is now trivially valid, because  if $\gamma\subset A\cap f^{-1}(A)$ is homotopically trivial in $A$, it is also
homotopically trivial in $A\cap f^{-1}(A)$ .

\section{Index calculation}

We recall the classical definition of Lefschetz index:

\begin{defi}
\label{index}
Let $\gamma: S^1 \to \R^2$ be a closed curve and $f$ a continuous map defined on $\gamma$ and without fixed points on $\gamma$. Define $I_f(\gamma)$, the Lefschetz index of $f$ in $\gamma$ as
the degree of the circle map $x \to \frac{f(\gamma (x))-\gamma(x)}{||f(\gamma (x))-\gamma(x)||}$.
\end{defi}

And the classical Lefschetz fixed point Theorem: 

\begin{thm}
Let $f$ be a continuous self-map of the plane. If $\gamma$ is a simple closed curve such that $I_f(\gamma)\neq 0$, then $f$ has a fixed point in the  bounded component of 
$\R^2\setminus\gamma$.

\end{thm}

We will need some techniques for calculating indexes of curves. The proof of the next lemma can be found in  \cite[Lemma 3]{iprx3}.

\begin{lemma}
\label{rectangulo}
Let $R\subset \R ^2$ be the square $[-1,1]^2$.  Let $f$ be defined on $\partial R$ such that:
\begin{itemize}
 \item $f(\{y=1\})\subset \{y>1\}$,
 \item $f(\{y=-1\})\subset \{y<-1\}$,
 \item $f(\{x=1\})\subset \{x>1\}$,
 \item $f(\{x=-1\})\subset \{x<-1\}$.

\end{itemize}

Then, $I_f (\gamma) = 1$, where $\gamma$ is $\partial R$ with the positive orientation.

\end{lemma}

Of course, the fact that $R$ is a square is not  essential in the  hypothesis. 

\begin{clly}
\label{torcido}
Let $\alpha$ and $\beta$ be disjoint simple proper lines in the plane, each one of which separate the plane. Let $\gamma$ and $\delta$ be another pair of disjoint curves separating
the plane. Assume also that each $\gamma$ and $\delta$ intersect $\alpha$ in one point and $\beta$ in one point. Now let $\Gamma$ be the simple closed curve determined by the four
arcs of the curves delimited by the intersection points, with the positive orientation.
Now let $f$ be a map defined on $\Gamma$ such that $f(\Gamma\cap\alpha)$ is contained in the component of the complement of $\alpha$ that does not contains $\beta$,
$f(\Gamma\cap\beta)$ is contained in the component of the complement of $\beta$ that does not contains
$\alpha$, that $f(\Gamma\cap\delta)$ is contained in the component of the complement of $\delta$ that does not contain $\gamma$ and that
$f(\Gamma\cap\gamma)$ is contained in the component of the complement of $\gamma$ that does not contain $\delta$. Then the index of $f$ in $\Gamma$ is equal to $1$.
\end{clly}

\begin{rk}\label{remarkrectangulo}
\begin{enumerate}

\noindent\item If the last two items (of Lemma \ref{rectangulo}) are changed to $f(\{x=1\})\subset \{x<1\}$ and $f(\{x=-1\})\subset \{x>-1\}$, then the conclusion is $I_f(\partial R)=-1$.

\item If the first two items (of Lemma \ref{rectangulo}) are changed to $f(\{y=1\})\subset \{y<1\}$,
 $f(\{y=-1\})\subset \{y>-1\}$, then the conclusion is is $I_f(\partial R)=-1$.

\item If \begin{itemize}
 \item $f(\{y=1\})\subset \{y<1\}$,
 \item $f(\{y=-1\})\subset \{y>-1\}$,
 \item $f(\{x=1\})\subset \{x<1\}$,
 \item $f(\{x=-1\})\subset \{x>-1\}$.
\end{itemize}
Then, $I_f (\gamma) = 1$, where $\gamma$ is $\partial R$ with the positive orientation.

\end{enumerate}
\end{rk}

Again, the fact that $R$ is a square is unimportant and therefore we obtain the corresponding corollary in each case, analogous to Corollary \ref{torcido}.

Let $A=\R^2\setminus \{(0,0) \}$ and $A^{'}\subset A$ be an essential subannulus whose boundaries are two simple closed curves $\gamma_1$ and $\gamma_2$ as in the Figure
\ref{figura11}.

\begin{figure}
\caption{}
\label{figura11}
\begin{center}

\psfrag{2}{$\partial A^{'}_{-}$}
\psfrag{1}{$\partial A^{'}_{+}$}
\psfrag{f}{$f$}
\psfrag{aa}{$A^{'}$}
\psfrag{a}{$A$}
\psfrag{f2}{$f(\partial A^{'}_{-})$}
\psfrag{f1}{$f(\partial A^{'}_{+})$}
\psfrag{f}{$f$}
\psfrag{0}{$(0,0)$}
\psfrag{gamma1}{$\gamma_1$}
\psfrag{gamma2}{$\gamma_2$}
\subfigure[]{\includegraphics[scale=0.21]{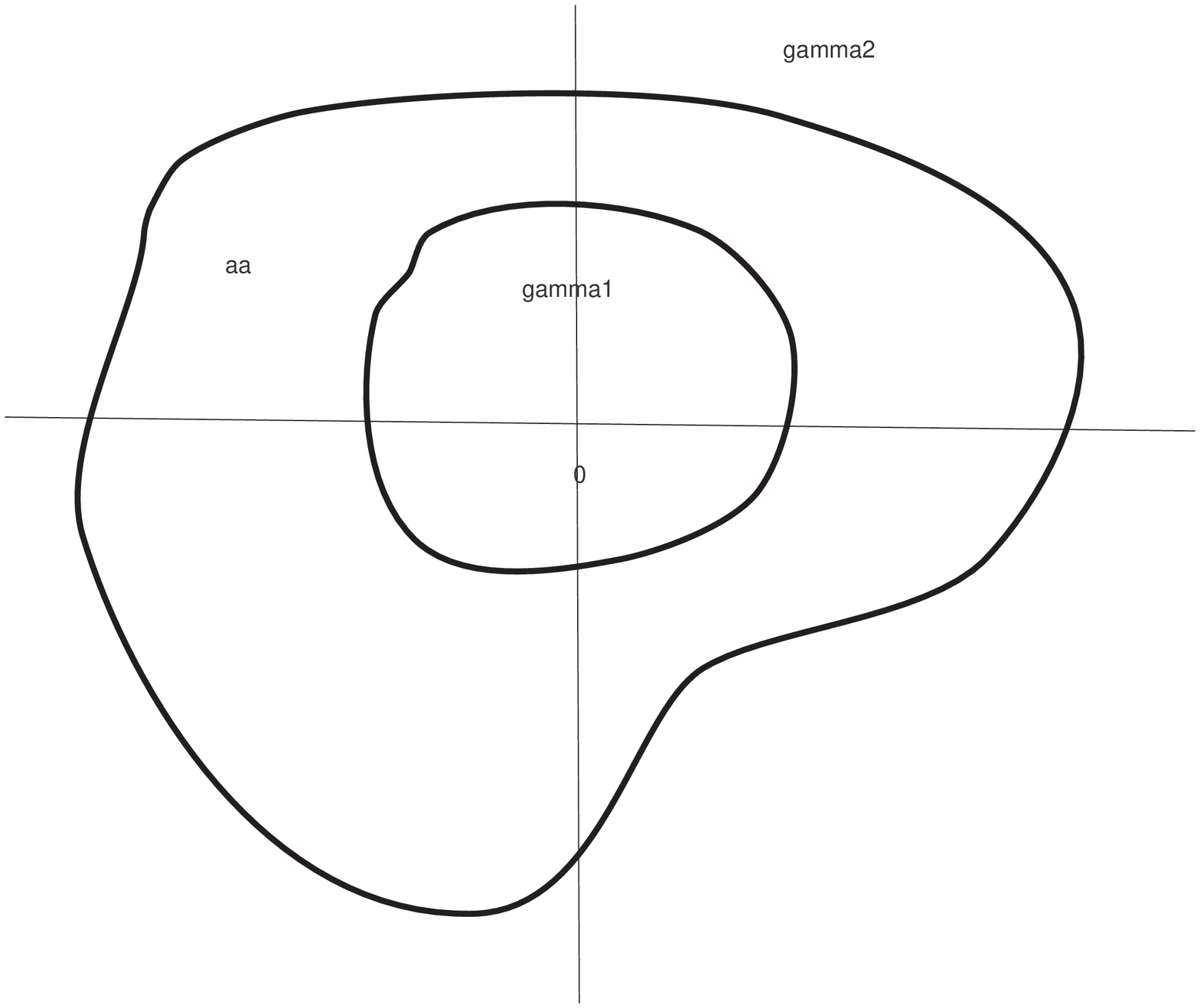}}
\end{center}
\end{figure}
 
Let $(\tilde{A},\Pi_{A} )$  and $(\tilde{A^{'}},\Pi_{A^{'}}) $ be  universal covering projections with $\tilde A '=\Pi_A^{-1}(A')  $ and $\Pi_{A'}: \Pi_A|_{\Pi_A^{-1}(A')}$.
Any $f:A^{'}\to A$ lifts to a map $F:\tilde{A^{'}}\to \tilde{A}$, and as  $\tilde{A^{'}}\subset \tilde{A}$ it makes sense to consider the fixed points of $F$.

The boundary components of $A^{'}$ will be denoted  $\partial A^{'}_{-}$ and $\partial A^{'}_{+}$, with $\partial A^{'}_{-}$ the boundary component which is closer to the origin.

We say that the annulus $A^{'}$ is {\it repelling} for  $f$ si $f(\partial A^{'}_{+})$ is contained in the unbounded connected component of $\R^{2}\setminus\partial A^{'}_{+}$ and 
$f(\partial A^{'}_{-})$ is contained in the bounded connected component of $\R^{2}\setminus\partial A^{'}_{-}$.

Our next result is to prove that a degree $d$ repelling annulus has at least  $|d-1|$ fixed points.  It is classical Nielsen theory that for finding $|d-1|$ different fixed points, it
is enough to find one fixed point for any given lift of the map to its universal covering space (see, for example,  Corollary 1 in \cite{iprx2}).

\noindent
\begin{thm} \label{t3}
Let $f: A^{'} \to A$ be a degree $d$ map. If $A^{'}$ is a repelling annulus, then $f$ has at least $|d-1|$ fixed points.

\end{thm}

 \begin{proof}
   As mentioned above, it is enough to show that any lift of $F$ to $\tilde A' \subset \tilde A$ has a fixed point.
   
  If $d>0$, we may assume $d\geq 2$ (if $d=1$ the conclusion of the theorem is trivially satisfied).
  Let  $F$ be a lift of $f$. We take $\tilde{A}=\{ (x,y)\in\R^{2}: \  0<y<1 \}$ and note that as the degree of $f$
  is $d$ we have the equality $F(x+1,y)=F(x,y)+(d,0)$ for all $x\in \tilde A'$. Let $\gamma_1=\partial A^{'}_{-}$, $\gamma_2=\partial A^{'}_{+}$ and 
  $\tilde{\gamma_{i}}=\Pi^{-1}_{A}(\gamma_i)$, $i=1,2$. Note that each $\tilde{\gamma_{i}}$ is an embedded line separating $\tilde{A}$ in two connected components.

Take a point  $q_{0}\in \tilde{\gamma_{1}}$ and let $\tilde{\gamma}_{1}^{0}$ be the lift of $\gamma_1$ starting at $q_{0}$  (see Figure \ref{figura5}).
For each $m\in\N$ let
$$ \Gamma_1^m=\bigcup_{k=-m}^m(\tilde{\gamma}_{1}^{0}+(k,0)    )  .    $$

 Let $V^{'}_0$ be a simple arc joining the ends of $\tilde{A}$ such that  $V^{'}_0\cap  \tilde{\gamma_{1}}=\{q_0\}$ y $V^{'}_0\cap  \tilde{\gamma_{2}}=\{p_0\}$ (see Figure \ref{figura5}).
 
 We let $\tilde{\gamma}_{2}^{0}$ be the lift of  $\gamma_2$ starting at $p_{0}$ and for each $m\in\N$ we let
$$ \Gamma_2^m=\bigcup_{k=-m}^m(\tilde{\gamma}_{2}^{0}+(k,0)    )   .   $$

 For each $m\in\Z$,  $V^{'}_m=V^{'}_0+(m,0)$. Let $V_0\subset V_{0}^{'}$ be the arc joining $q_0$ and $p_0$ and let $V_m=V_0+(m,0)$.

 We let $\beta$ be the simple closed and positively oriented loop defined by:
 \begin{equation}\label{curvacerrada}
   \beta=               \Gamma_2^m    . V_m . \Gamma_1^m .  V_{-m}
 \end{equation}

Note that
as $F(x+m,y)=F(x,y)+(dm,0)$ for all $m\in \Z$, and $d>1$, then there exists $m>0$ such that the following conditions hold:

\begin{enumerate}
\item
The image of $V_m$ under $F$ is contained in the connected component of $\R^2\setminus V^{'}_m$ that does not contain $V^{'}_0$.
\item
The image of $V_{-m}$ under $F$ is contained in the connected component of $\R^2\setminus V^{'}_{-m} $  that does not contain $V^{'}_0$.

\item As $A^{'}$ is repelling, $F(\Gamma_1^m)$ is in the connected component of  $\tilde{A}\setminus \tilde{\gamma}_1$ that does not contain $\tilde{\gamma}_2$ and
$F(\Gamma_2^m)$ is in the connected component of   $\tilde{A}\setminus \tilde{\gamma}_2$ that does not contain$\tilde{\gamma}_1$.

\end{enumerate}
 Then, by Corollary \ref{torcido} we obtain $I_F(\beta)=1$ and therefore $F$ has a fixed point.\\
 
 If $d\leq -1$, considering the same loop $\beta$, we may find $m\in\Z$ such that:

\begin{enumerate}
\item
The image of $V_m$ under $F$ is contained in the connected component of $\R^2\setminus V^{'}_m$ containing $V^{'}_0$. 
\item
The image of $V_{-m}$ under $F$ is contained in the connected component of $\R^2\setminus V^{'}_{-m} $  containing $V^{'}_0$.

\item $F(\Gamma_1^m)$ is in the connected component of  $\tilde{A}\setminus \tilde{\gamma}_1$ that does not contain $\tilde{\gamma}_2$ and
$F(\Gamma_2^m)$ is in the connected component of   $\tilde{A}\setminus \tilde{\gamma}_2$ that does not contain$\tilde{\gamma}_1$.
\end{enumerate}

 Now by Remark \ref{remarkrectangulo} item (1) ( and its corresponding corollary) we obtain $I_F(\beta)=-1$ and therefore $F$ has a fixed point.\\

If $d=0$, then $F(x+n,y)=F(x,y)$, for all  $n\in \Z$. It follows that by considering the same loop $\beta$, one may find $m\in \Z$ verifying the same properties than for the case $d<0$.  
Again, this gives a fixed point for $F$.

\end{proof}

\begin{figure}[h]
\psfrag{q1}{$q_{0}$}

\psfrag{p1}{$p_{0}$}

\psfrag{gamma02}{$\tilde{\gamma}_{1}^{0}$}

\psfrag{q2}{$p_{0}$}
\psfrag{beta}{$\beta$}
\psfrag{gamma01}{$\tilde{\gamma}_{2}^{0}$}
\psfrag{v-m}{$V_{-m}$}
\psfrag{v0}{$V_{0}$}
\psfrag{vm}{$V_{m}$}
\psfrag{vv0}{$V^{'}_{0}$}
\psfrag{aa}{$0$}
\psfrag{bb}{$1$}
\psfrag{i0}{$I_0$}\psfrag{i1}{$I_1$}
\psfrag{i2}{$I_{2}$}
\psfrag{ii}{$I$}
\psfrag{vvm}{$V^{'}_{m}$}
\begin{center}
\caption{\label{figura5}}
\subfigure{\includegraphics[scale=0.25]{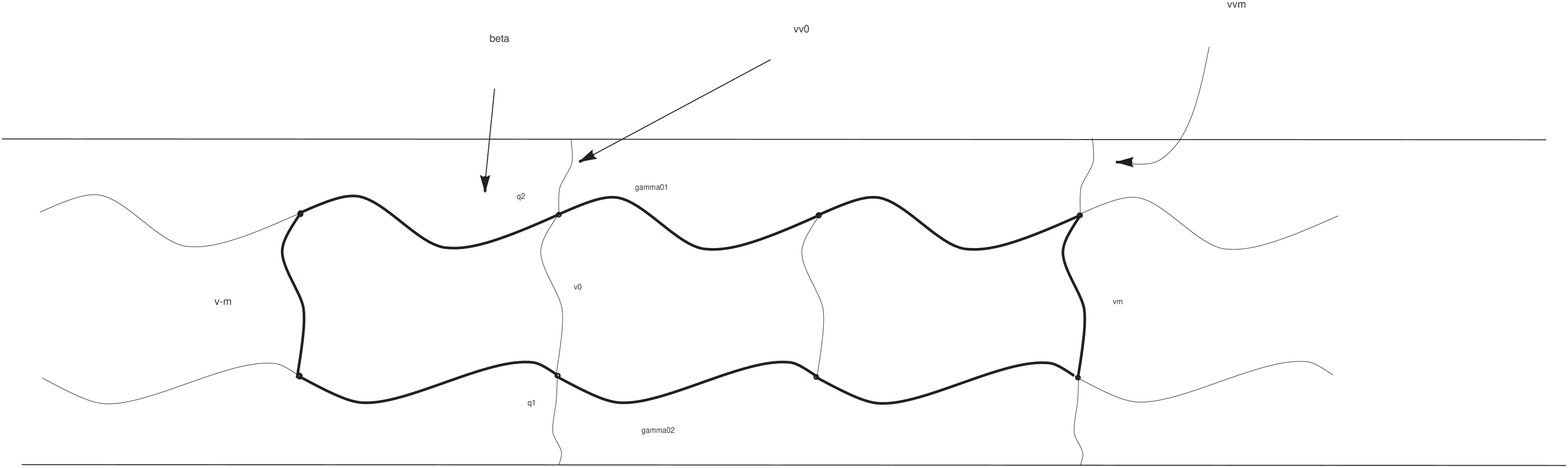}}
\end{center}
\end{figure}

\begin{figure}
\caption{}
\label{figura1}
\begin{center}
\psfrag{0}{$(0,0)$}
\psfrag{aa}{$A^{'}$}
\psfrag{2}{$\partial A^{'}_{-}$}
\psfrag{1}{$\partial A^{'}_{+}$}
\psfrag{n}{$N$}
\psfrag{s}{$S$}
\psfrag{nn}{$N^{'}$}
\psfrag{ss}{$S^{'}$}
\psfrag{a}{$A$}
\psfrag{f2}{$f(\partial A^{'}_{-})$}
\psfrag{f1}{$f(\partial A^{'}_{+})$}
\psfrag{f}{$f$}
\psfrag{1}{$\delta_1=1$}
\psfrag{2}{$\delta_2=0$}
\psfrag{3}{$\delta_3=1$}
\psfrag{fn}{$f(N^{'})=N \mbox{ and } f(S^{'})=S.$}
\subfigure[]{\includegraphics[scale=0.21]{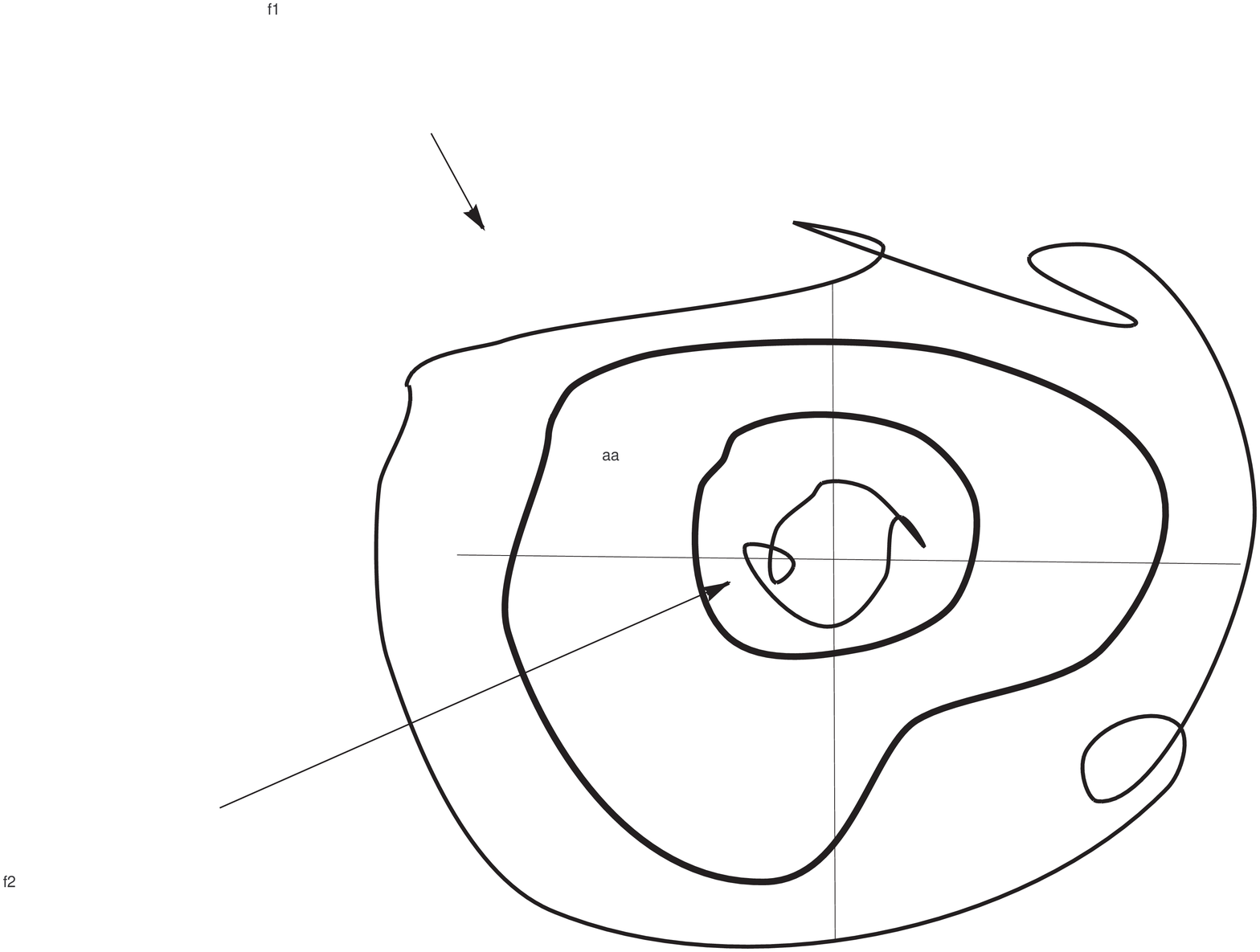}}
\subfigure[]{\includegraphics[scale=0.23]{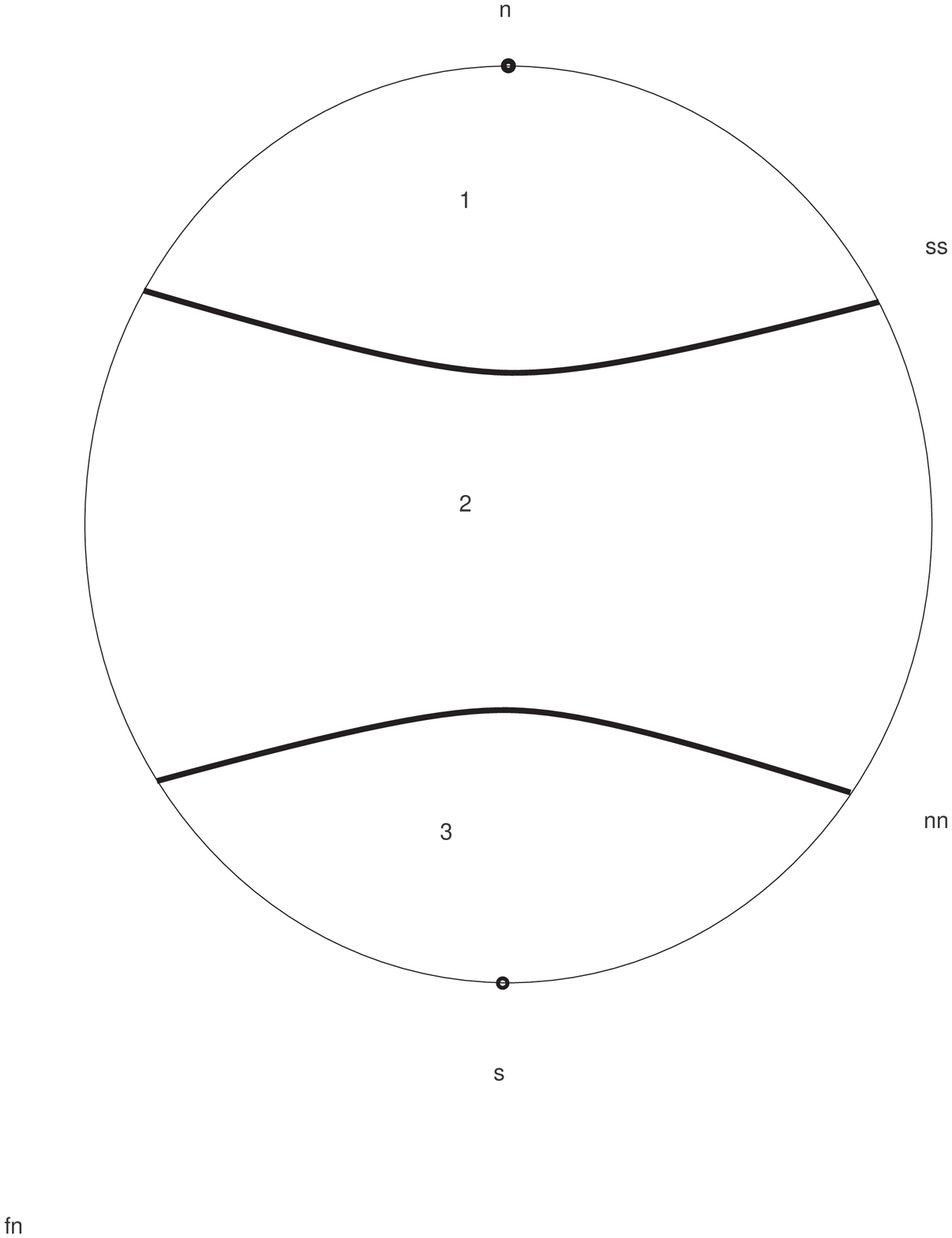}}
\end{center}
\end{figure}

As an application, we obtain a new proof of Pugh and Shub's Theorem in \cite{ps}:

\begin{thm} Let $f:S^2\to S^2$ be a degree two endomorphism that preserves the latitude foliation. If $f$ is $C^1$ then $f^n$ has at least $2^n$ fixed points.

\end{thm}

And of Misurewicz's Theorem in \cite{mis}:

\begin{thm}  Let $f:S^2\to S^2$ be a degree two longitudinal $C^1$ endomorphism. Then $f^n$ has at least $2^n$ fixed points.
 
\end{thm}

Note that in both cases one obtains a repelling annulus. (If $f$ is $C^1$, then $N$ and $S$ if fixed, must be attracting.  If they are not fixed they interchange as a result
of their hypothesis of preserving either the latitudinal or longitudinal foliation.  This also gives a repelling annulus).

More generally, we have the following result:

\begin{clly} Let $f:S^2\to S^2$ be a degree $d$ $C^1$ endomorphism.  If $f^{-1}(N)\subset \{N,S\}$ and $f^{-1}(S)\subset \{N,S\}$ then $f$ has at least $|d-1|$ fixed points.
 
\end{clly}

\section{Degree}

Let $\sigma: \Sigma_1 \to \Sigma_2$ be a continuous map between topological 2-spheres $\Sigma_1$ and $\Sigma _2$ and choose generators $[\alpha_i]$ of $H_ 2(\Sigma_i)\approx \Z$ so that $\deg \sigma$ is 
defined.  Suppose that for some $y\in \Sigma_2$, the preimage $\sigma^{-1}(y)$
consists of finitely many points $x_1, \ldots, x_n$.  Let $U_1, \ldots, U_m$ be disjoint neighbourhoods of these connected components, mapped by 
$\sigma$
into a neighbourhood $V$ of $y$.  Then $\sigma(U_i-x_i)\subset V-y$ for each $i$, the groups $H_2 (U_i, U_i-x_i)$ and $H_2 (V, V-y)$ can be identified with 
$H_2(\Sigma_1)$ and $H_2(\Sigma_2)$ respectively and there exists an integer called the {\it local degree} of $\sigma$ at $x_i$, written 
$\deg \sigma|_{x_i}$ such that the homomorphism $\sigma_*:H_2 (U_i, U_i-x_i)\to H_2 (V, V-y)$ maps $[\alpha_1]$ to $\deg \sigma|_{x_i} [\alpha_2]$ .  Moreover, $\deg \sigma = \sum_i \deg \sigma|_{x_i}$ (See \cite{hatcher} Proposition 2.30). Note also that the groups $H_2 (U_i, U_i-x_i)$ and $H_2 (V, V-y)$
can be identified with $H_1(U-x_i)$ and $H_1(V-y)$ respectively as they fit in the long exact sequence of homology groups:
$$\ldots \longrightarrow H_2 (U) \longrightarrow H_2 (U, U-x_i) \longrightarrow H_1 (U-x_i) \longrightarrow H_1 (U) \longrightarrow \ldots ,$$\noindent where the first and last group
vanish and so the middle map is an isomorphism. So, the corresponding degree of the annulus map
$\sigma|_{U-x_i}: U-x_i \to V-y$ is also $\deg \sigma|_{x_i}$.

Recall that we may assume that the set of connected components of $f^{-1} (N\cup S)$ is $\{N, S, X_1, \ldots, X_p\}$, where each $X_i$ is an essential circle in $A$.

The topological space obtained by collapsing the connected components of $f^{-1} (N\cup S)$ to points is homeomorphic to the sum of a finite number of spheres that we denote 
$\vee_i \Sigma_i$ (although formally it is not a wedge sum). More precisely, $\vee_i \Sigma_i = S^2/\sim$ where $x\sim y$ if and only if $x$ and $y$ belong to the same connected
component of $f^{-1} (N\cup S)$ (coloquially, a cactus, see Figure \ref{cac}).

\begin{figure}[ht]

\begin{center}
{\includegraphics[scale=0.23]{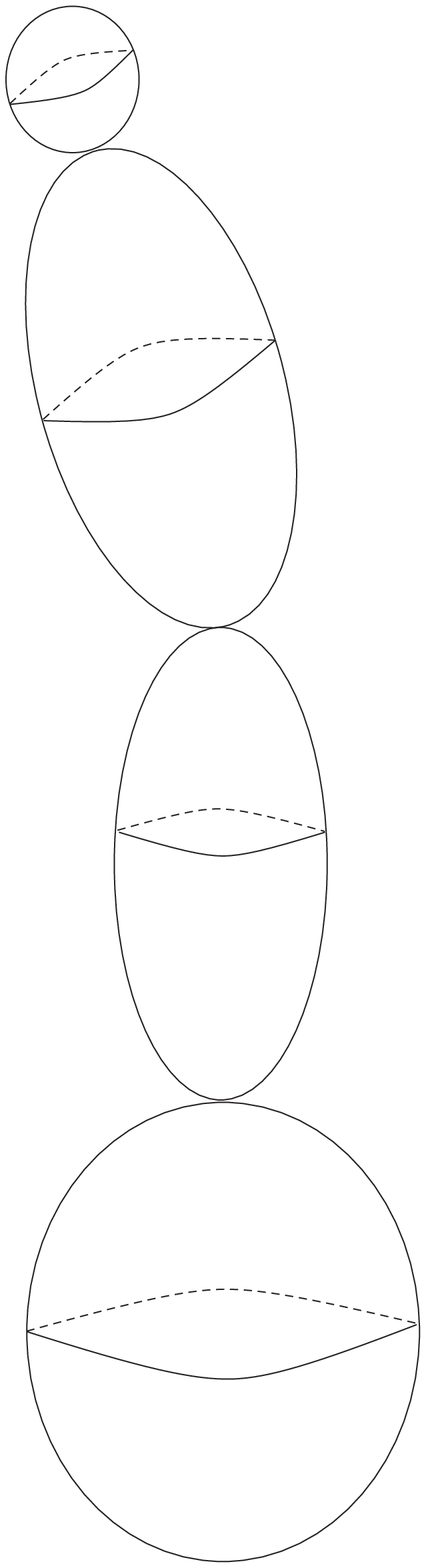}}

\caption{A cactus}\label{cac}
\end{center}
\end{figure}

If $q: S^2\to \vee _i \Sigma _i$ is the quotient projection, then the induced map $q_*: H_2 (S^2)\to H_2 (\vee _i \Sigma _i)$ is surjective, and we may
choose generators of each $H_2 (\Sigma _i)$ such that $q_* (1) = (1, 1, \ldots, 1)$.  Note that $f$ factors to $\hat f: \vee _i \Sigma _i \to S^2$ and induces a morphism 
$\hat f_*:\oplus _i H_2 (\Sigma _i)\to H_2 (S^2)$. We define the degree on each sphere $\Sigma _i$ that we denote $\deg f |_{\Sigma_i}$ as $\hat f_* (e_i)$, where $e_i$ is the 
$i$-th cannonic vector.  In other words, $\deg f |_{\Sigma_i}$ is the degree of the map $\sigma_i=\hat f|_{\Sigma _i} : \Sigma _i \to S^ 2$.

\begin{figure}[ht]
\psfrag{f}{$f$}
\psfrag{b}{$q$}
\psfrag{c}{$\hat f$}
\psfrag{n}{$N$}
\psfrag{s}{$S$}
\begin{center}
{\includegraphics[scale=0.23]{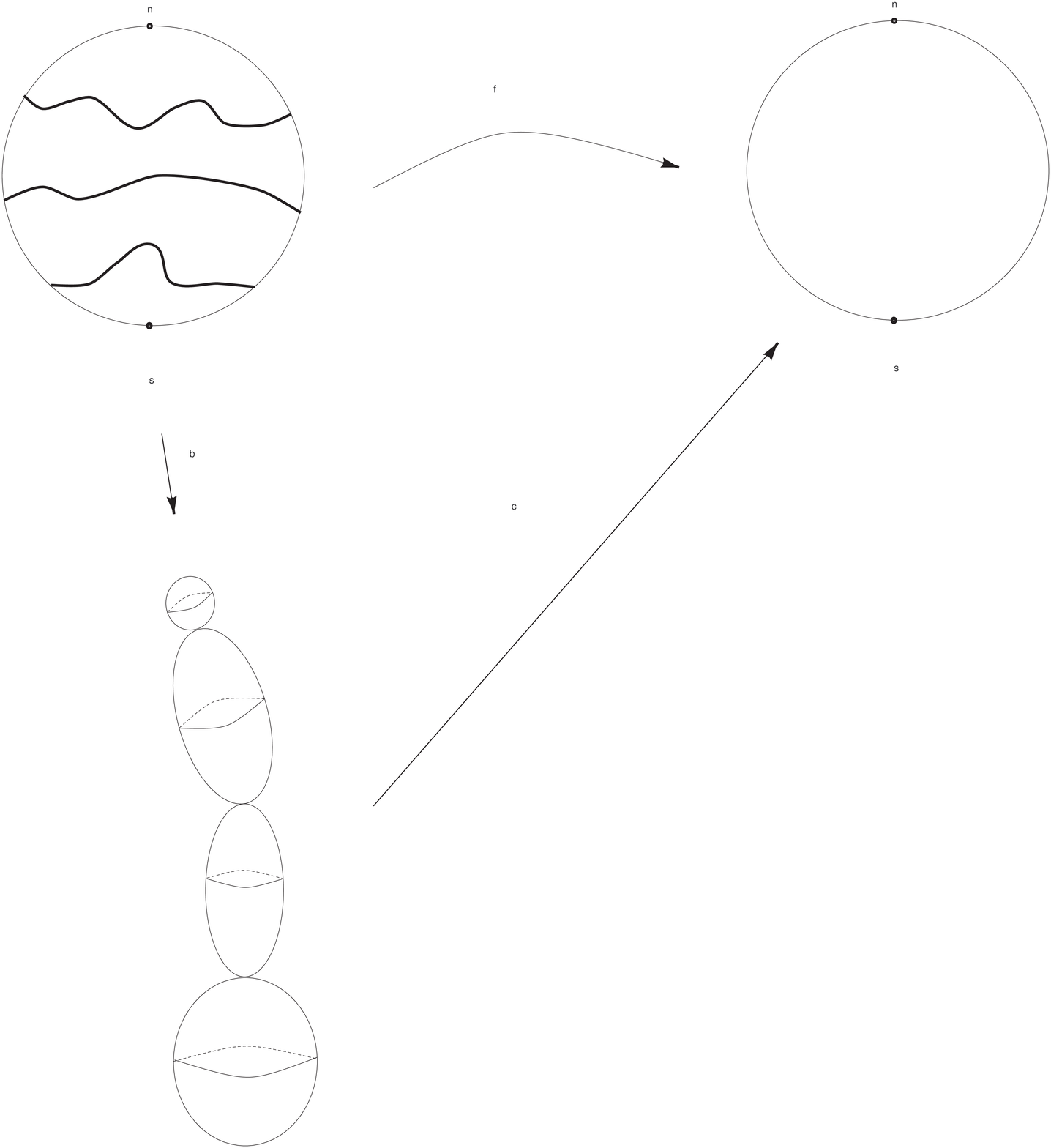}}

\caption{}
\end{center}
\end{figure}

\begin{lemma}\label{suma} $\deg f = \sum _i \deg f |_{\Sigma _i}$
 
\end{lemma}

\begin{proof}  By definition, $\deg f = f_* (1)$ and $f_* = \hat f _* q_*$, so $f_* (1)= \hat f _* q_* (1) = \hat f_* (1, 1, \ldots, 1) = \sum _i \hat f_* (e_i) =
\sum _i \deg f |_{\Sigma _i}$.
 
\end{proof}
For all $i$ we denote $N_i = \sigma_i ^ {-1} (N)$ and $S_i = \sigma_i ^ {-1} (S)$ and note that $\deg \sigma_i = \deg \sigma_i |_{N_i} = \deg \sigma_i |_{S_i}$.

There is a natural identification $H_2(S^2)\approx H_1 (A)$:
the long exact sequence for the pair $(S^2-N, A)$,

$$\ldots \longrightarrow H_2 (A) \longrightarrow H_2 (S^2-N) \longrightarrow H_2 (S^2-N, A) \longrightarrow H_1 (A) \longrightarrow H_1(S^2-N)\longrightarrow \ldots ,$$\noindent 
gives an isomorphism $H_2 (S^2-N, A) \approx H_1 (A)$ while by excision $ H^2(S^2, S^2-S)\approx H_2(S^2-N, S^2-S-N) = H_2(S^2-N, A) $ where $ H_2(S^2, S^2-S)\approx H_2 (S^2)$ again by
the long exact sequence of pairs.

Note that for each $i$, $U_i = q^ {-1} 
(\Sigma _i \backslash \{N_i, S_i\})$ is an open essential subannulus of $A$ and one may choose a generator of $H_1(U_i)$ such that the morphism induced on homology by the inclusion 
$U_i\hookrightarrow A$
is the identity. Then, choosing these induced generators of $H_1(A)$ and $ H_1(U_i)$
$f|_{U_i}: U_i \to A$ defines an 
{\it annular degree} $\delta_i$.

\begin{lemma}\label{delta} $\delta _i = \pm \deg \sigma_i$
 
\end{lemma}

\begin{proof}  Recall that $\deg \sigma_i = \deg \sigma_i |_{S_i} = \pm \delta _i$ where the change of sign may occur if the generator of  $H_1(U-S_i)\approx H_2 (\Sigma _i)$ has the 
opposite sign as that of $H_1 (U_i)$.
(Here we have taken $U$ a neighbourhood of $S_i$ on $\Sigma _i$ as in the definition of local
degree).

\end{proof}

\section{Proof}

We devote this section to the proof of:\\

{\bf Theorem A.}
{\em Let $f:S^2\to S^2$ be a continuous map such that $\deg f = d, |d|>1$.  Suppose $f$ has two attracting fixed points denoted $N$ and $S$ and let $A=S^2\setminus \{N,S\}$.
Assume that if a loop $\gamma\subset f^{-1}(A)$ is homotopically trivial in $A$, then $f(\gamma)$ is also 
homotopically trivial in $A$.
Then $f$ has the rate.}\\

\begin{proof}
 Let $d_i = \deg \sigma_i$.  The content of Lemma \ref{suma} is $\deg f = \sum _i d_i$, and Lemma \ref{delta} gives $\delta _i = \pm d_i$.  Furthermore, Theorem \ref{t3} gives
 $\sum _i |\delta _i -1| \leq \# \fix (f)$.  Suppose $d_i>|\delta_i-1|$.  It follows that $d_i=\delta_i >0$ and $d_{i+1} = -\delta _{i+1}$.  We claim that $|\delta_{i+1}-1|> d_{i+1}$:  if
 $d_{i+1}<0$ this holds trivially, and if $d_{i+1}\geq 0$, then $\delta_{i+1}\leq 0$ which implies $|\delta_{i+1}-1|> d_{i+1}$.
 We conclude that $d_i+d_{i+1}\leq |\delta_i-1| + |\delta_{i+1}-1|$.  
 
 To finish the proof,  note that every positive iterate of $f$ verifies the  hypothesis of the theorem to  obtain $d^n \leq \# \fix (f^n)$.
\end{proof}

\end{document}